\documentclass[11pt, a4paper, reqno]{amsart}
\usepackage[a4paper, total={6in, 9in}]{geometry}
\usepackage{amsfonts}   %% LaTeX2e document.
\usepackage{amsmath}
\usepackage[utf8]{inputenc}
\usepackage[english]{babel}
\usepackage{amssymb}
\usepackage{tikz-cd}
\usepackage{comment}
\usepackage{amsthm}
\usepackage{xcolor}
\usepackage[normalem]{ulem}
\newtheorem{theorem}{Theorem}[section]

\newtheorem{claim}{Claim}

\newtheorem*{claim*}{Claim}
\newtheorem*{proofclaim*}{Proof of Claim}

\theoremstyle{theorem}

\newtheorem*{remark}{Remark}
\usepackage{embedfile}

\newtheorem*{Exercise 11.3.B}{Exercise 11.3.B}

\newtheorem{conjecture}{Conjecture}
\numberwithin{equation}{section}

\newcommand{\cm}{\text{Cohen-Macaulay}}

\usepackage{fancyhdr}

\usepackage{graphicx}
\graphicspath{ {./images/} }

\usepackage{hyperref}                                    %! for more notes see the bibliography section below
\usepackage{enumerate}

\usepackage[all,cmtip]{xy}

\usepackage{blindtext}
\begin{document}
	\newcommand{\fg}{\text{finitely generated}}	
	\newcommand{\les}{\text{long exact sequence}}

	\title[An Instance of the Small Cohen-Macaulay Conjecture]
	{On an Instance of   the Small Cohen-Macaulay Conjecture}
	\author{Likun Xie}

	\address{Department of Mathematics, University of Illinois, 1409 West Green
		Street, Urbana, IL 61801, USA} \email{likunx2@illinois.edu}

	\begin{abstract} 
		We provide a simplified proof of a theorem proved by Tavanfar and Shimomoto which states that a quasi-Gorenstein deformation  of a $3$-dimensional quasi-Gorenstein local ring $(A,m,k)$ with $H^2_m(A)=k$  admits a small Cohen-Macaulay module. 
	\end{abstract}
	\maketitle
	%\begin{center}
	%	\Large	\textbf{Proof of Mcintosh's Conjecture}
	%	\large		
	%	\normalsize
	%	\end{center}
%	\begin{center}		
%	\today
%	\end{center}	
\section{Introduction}

In this paper, we consider a commutative local ring $(A,m,k)$ that is  a homomorphic image of a Gorenstein ring such that $\operatorname{dim} A=n$ and $\operatorname{depth} A=n-1$. S. P. Dutta pointed out the following fact in a private conversation with the author: Let $R$ be a Gorenstein local ring with $\operatorname{dim} R=\operatorname{dim} A=d$ and $A\cong R/I $ is a homomorphic image of $R$. Assume that $R$ and $A$ are complete. 

From the short exact sequence $$0\to I\to R\to A\to 0,$$ we have that  $\operatorname{depth} I=d$. 

After applying $\operatorname{Hom}_R(-, R) $, we have the exact sequence where $ I^*=\operatorname{Hom}_R(I,R)$ $$0\to A \to R\xrightarrow{f} I^* \to \operatorname{Ext} _R^1(A,R)\to 0$$  

Let $T_2$ be the kernel of $f$, and let $T_3$ denote $\operatorname{Ext} _R^1(A,R)$. Since $I$ is  maximal \cm, $I^*=\operatorname{Hom}_R( I,R) $ is also maximal \cm. Hence $\operatorname{depth} I^*=d$. Therefore, we have $\operatorname{depth }T_2= d-2$ and $\operatorname{depth } T_3= d-3$. 
Moreover,  $$\operatorname{Ext }_R^1(A,R)
\cong \operatorname{Ext} _R^2(T_2,R) \cong \operatorname{Ext}_R^3(T_3, R ),$$ thus  $\operatorname{grade}_R \operatorname{Ext} _R^1(A,R)\geq 3$. Therefore, 
\begin{equation}
	\label{intro_eq}
	\operatorname{dim} \operatorname{Ext} _R^1(A,R)\leq d-3.
\end{equation}
 Let $x_1,x_2,x_3$ be a regular sequence on $R$  in $\operatorname{ann}_R \operatorname{Ext}_R^1(A,R)$, then  $\operatorname{Ext}_R^1(A,R)$ is  a small Cohen-Macaulay module over $A/(x_1,x_2,x_3)$. 

Recently, a new instance of the existence of small Cohen-Macauly module was proved in Theorem 3.2, \cite{remark}. The main focus of this paper is to give a simplified proof of this theorem. We reformulate the theorem as follow. 

	\begin{theorem}
	Let $(A,m,k)$ be a quasi-Gorenstein local ring which is a homomorphic image of a Gorenstein ring. Suppose that $\operatorname{depth} A= \operatorname{dim} A-1$,  and there exists  a regular sequence $\underline{y}\in A$ such that $A/(\underline{y})$ is a $3$-dimensional  quasi-Gorenstein ring and $H_m^2(A/(\underline{y}))=k$. 
	%Suppose that there exists a (possibly nonempty) regular sequence $\underline{y}$ of $A$ such that $A/(\underline{y}) A$ is quasi-Gorenstein and Buchsbaum of dimension $3$ and $H_m^2(A/(\underline{y})A) =k$
	Then $A$ admits a maximal Cohen-Macaulay module. %Namely, there is some $x\in (0:_A H_m^{dimA-1}(A)) $ such that $x,\underline{y}$ forms a regular sequence   of $A$. The first syzygy $\Omega_{A/xA}$ of the canonical module $\omega_{A/xA}$ in the minimal free resolution is a maximal Cohen-Macaulay module of rank $2$. 
\end{theorem}

The Small Cohen-Macaulay Conjecture states the following, see   \cite[Section 2]{conjecture}:
\begin{conjecture}
	Every complete local domain has a small Cohen-Macaulay module.
\end{conjecture}

The above theorem tackles a special case in this conjecture. 

 Both our proof and the proof in \cite{remark} reduce to the case of dimension $3$ first. In our approach, a key fact that we used is that   $\operatorname{Ext}_R^1(A,R) $ is \cm\space of dimension $\operatorname{dim}A-3$ and that the exact sequence $\underline{y}\in A$ is regular on $\operatorname{Ext}_R^1(A,R)$, for $A$ a quasi-Gorenstein local ring  with $\operatorname{depth}A=\operatorname{dim}A-1$ which is a homomorphic image of a Gorenstein ring $R$. This  quickly gives the liftability of $\omega_{A/(x,\underline{y_i})}$ to $A/(x,\underline{y_{i-1}})$ which is desired for the reduction. We did not use the theory of attached primes of local cohomology and the recent results studied by Tavanfar and Tousi on quasi-Gorenstein rings   as in \cite{remark}. We remarked that in the reduction step, we do not yet need the assumption $H_m^2(A/(\underline{y}))=k$.  %Indeed  $H_m^2(A/(\underline{y}))$  being minimally generated by one element is necessary for the module $\Omega_{A/(x,\underline{y})}$ to be maximal \cm. 
 The rest of our proof centres around the five-term exact sequence  \eqref{five_term_exact sequence} which makes it clear that the aim is to show  $\operatorname{Ext}_A^2(\omega_{A/xA},A)=0$. And we did not  use the theory of $S_2$-ification and dualizing complex in derived category as in \cite{remark}.

 	Our proof features  straightforward arguments  that is easy to navigate, without relying on heavy machinery.  It is clear from the proof where $H_m^2(A/(\underline{y}))=k$ is essential for Theorem \ref{main_theorem} to be true: to show that $j$ is surjective by showing that $i$ is not surjective in \eqref{exact1}. Indeed, when $H_m^2(A/(\underline{y}))=k^n$ for any $n>1$, $\Omega_{A/xA}$ fails to be maximal \cm. To see this, assume $H_m^2(A/(\underline{y}))=k^n$ for some $n>1$ while the other assumptions in the theorem remain the same. After reduction to $\operatorname{dim }A =3$, an analogue of  \eqref{short_exact_1} carries through to give a short exact sequence $$0\to A/xA\to \omega_{A/xA}\to k^n\to 0,$$
 	 and then applying $\operatorname{Hom}_{A/xA}(k^n,-)$ gives $$\operatorname{Ext}^1_{A/xA} (k^n,A/xA)\cong \operatorname{Hom}_{A/xA} (k^n,k^n)\cong k^{n^2}.$$
 	 Following the line of the rest of the proof, we have $\Omega_{A/xA}$ is \cm, if and only if, $A/xA\xrightarrow{j} \operatorname{Ext}_{A/xA}^1(k^n,A/xA)\cong k^{n^2}$ is surjective. But this is impossible by tensoring $j$ with $-\otimes_{A/xA} k$. 
 	  %Moreover, this shows that for $\Omega_{A/xA}$ to be \cm, $H_m^2(A/(\underline{y}))$ must be minimally generated by $\leq 1$ element.
 	   
 	  More generally, it is natural to speculate whether $k$ in the condition $H_m^2(A/(\underline{y}))=k$ of Theorem \ref{main_theorem} can be replaced by other nontrivial modules. Suppose $H_m^2(A/(\underline{y}))=M\neq 0$ and all other conditions in Theorem \ref{main_theorem} remain the same. One can see  during the reduction step, $M$ must be of finite length. After reduction to $\operatorname{dim }A=3$, an analogue of \eqref{short_exact_1} gives the exact sequence $$0\to A/xA\to \omega_{A/xA}\to M^\vee\to 0,$$ then applying $\operatorname{Hom}_{A/xA} (M^\vee,-)$ gives  $$\operatorname{Ext} _{A/xA}^1(M^\vee, A/xA) \cong \operatorname{Hom}_{A/xA} (M^\vee,M^\vee). $$
 	  
 	  The map $j$ becomes $\operatorname{Hom}_{A/xA} (A/xA,A/xA) \to \operatorname{Ext}^1_{A/xA} (M^\vee, A/xA)$. Tracing backwards in the proof,  $\Omega_{A/xA}$ is maximal \cm, if an only if,  $j$ is surjective. For $j$ to be surjective, there must be $\operatorname{Hom}_{A/xA} (M^\vee,M^\vee )\otimes_{A/xA} k\cong k$, $M^\vee $ hence $M$ must be indecomposable.

 	 % And from what we just mentioned, $H_m^2(A/(\underline{y}))$ must be minimally generated by $\leq 1$ element.
 	  %So if $k$ in $H_m^2(A/(\underline{y}))=k$ is to be replaced by some other nontrivial module so that the  conclusions of Theorem \ref{main_theorem} hold, it is only possible that
 	  % $H_m^2(A/(\underline{y}))\cong A/J$ for some ideal $J$ such that $A/J$ is of finite length. 

\section{}

The following theorem is a reformulation of Theorem 3.2, \cite{remark}. For any local ring $(A,m,k)$ and an $A$-module $M$, we use the notation $\Omega_M$ to denote the first syzygy in the minimal free resolution of the canonical module $\omega_M$ of $M$.  We denote by $E=E_A(k)$ the injective hull of the residue field $k$ and by  $(-)^{\vee}=\operatorname{Hom}_A(-,E)$ the Matlis dual. We say that a local ring $A$ is \textit{quasi-Gorenstein} if the canonical module $\omega_A\cong A$.

	\begin{theorem}\label{main_theorem}
	Let $(A,m,k)$ be a quasi-Gorenstein local ring which is a homomorphic image of a Gorenstein ring. Suppose that $\operatorname{depth} A= \operatorname{dim} A-1$,  and there exists  a regular sequence $\underline{y}\in A$ such that $A/(\underline{y})$ is a $3$-dimensional  quasi-Gorenstein ring and $H_m^2(A/(\underline{y}))=k$. 
	%Suppose that there exists a (possibly nonempty) regular sequence $\underline{y}$ of $A$ such that $A/(\underline{y}) A$ is quasi-Gorenstein and Buchsbaum of dimension $3$ and $H_m^2(A/(\underline{y})A) =k$
	 Then $A$ admits a maximal Cohen-Macaulay module. Namely, there is some $x\in (0:_A H_m^{\operatorname{dim}A-1}(A)) $ such that $x,\underline{y}$ forms a regular sequence   of $A$. The first syzygy $\Omega_{A/xA}$ of the canonical module $\omega_{A/xA}$ in the minimal free resolution is a maximal Cohen-Macaulay module of rank $2$. 
	\end{theorem}
	\begin{proof}
		Let $d=\operatorname{dim} A=\operatorname{depth} A+1$, 
	 from the introduction, we have that $\operatorname{Ext} _R^1(A,R)$ is a \cm\space module over $R$ and therefore it is also a \cm\space $A$-module of dimension $d-3$.  Let $\underline{y} $ be the given regular sequence of length $d-3$ such that $A/(\underline{y})$ is quasi-Gorenstein.
We have $$\operatorname{Ext}_R^1(A,R) /\underline{y} \operatorname{Ext}_R^1(A,R)\cong \operatorname{Ext}_R^{d-2} (A/(\underline{y}), R) \cong  \operatorname{Ext}_{R/(\underline{y})}^1 (A/(\underline{y}), R/(\underline{y})).$$ 
Note that  $\operatorname{dim} \operatorname{Ext}_{R/(\underline{y})}^1 (A/(\underline{y}), R/(\underline{y}))=0$ by \eqref{intro_eq}, $\underline{y}$ is a system of parameter on  $\operatorname{Ext}_R^1(A,R)$ hence it is regular on $\operatorname{Ext}_R^1(A,R)$. 

Note that $A/(\underline{y})$ is $S_2$ and catenary, hence it's unmixed. Therefore, for any $p\in \operatorname{Ass} (A/(\underline{y}))$, $A_p$ is Cohen-Macaulay. 
%Now we show that $ann (\operatorname{Ext}_R^1(A,R) )\nsubseteq  \bigcup_{p\in Ass(A/(\underline{y}))} p$. Suppose not, then $ann (\operatorname{Ext}_R^1(A,R) )\subseteq p$ for some ${p\in Ass(A/(\underline{y}))}. $ Since $A/(\underline{y})$ is $S_2$ and catenary, we have $A/(\underline{y})$ is unmixed, therefore $A_p$  is \cm\space for any $p\in Ass(A/(\underline{y}))$. 
Take any $p\in \operatorname{Ass}(A/(\underline{y})) $ and let $p'$ be the inverse image of $p$ in $R$. Then $$(\operatorname{Ext}^1_R(A,R))_p\cong \operatorname{Ext}^1_{R_{p'}}(A_p,R_{p'}) \cong H_{{p'}R_{p'}}^{\operatorname{dim}R_{p'}-1}(A_p)^\vee\cong H_{A_{p}}^{\operatorname{dim}A_{p}-1}(A_p)^\vee=0.$$ 
Hence $\operatorname{ann} (\operatorname{Ext}_R^1(A,R) )\nsubseteq  \bigcup_{p\in \operatorname{Ass}(A/(\underline{y}))} p$, there exists $x\in \operatorname{ann}(\operatorname{Ext}_R^1(A,R))$ such that $x,\underline{y}$ is a regular sequence on $A$. 

Let $\underline{y_k}$ denote the sequence $y_1,\dots, y_k$ for $1\leq k\leq d-3$. 

Consider the short exact sequence $$0\to A/(x,\underline{y_{i-1}})\xrightarrow{y_{i}}A/(x,\underline{y_{i-1}})\to A/(x,\underline{y_{i}})\to 0, $$ after taking the long exact sequence of local cohomology and applying $\operatorname{Hom}_R(-,E)$, we have that $y_i$ is regular on $\omega_{A/(x,\underline{y_{i-1}})}$ and  the exact sequence 
\begin{align}\label{exact_2}
0\to \omega_{A/(x,\underline{y_{i-1}})}/y_i\omega_{A/(x,\underline{y_{i-1}})}\to \omega_{A/(x,\underline{y_{i-1}})}\to \operatorname{Ext} _R^{i+1}(A/(x,\underline{y_{i-1}}),R) \xrightarrow{y_i} \operatorname{Ext} _R^{i+1}(A/(x,\underline{y_{i-1}}),R) .
\end{align}

Note that the last map $y_i$ is injective since  $\underline{y} $ is a regular sequence  on $\operatorname{Ext} _R^1(A,R)$ and
$$\operatorname{Ext} _R^{i+1}(A/(x,\underline{y_{i-1}}),R)\cong 	\operatorname{Ext} _R^1(A,R)/(x,\underline{y_{i-1}})\cong \operatorname{Ext} _R^1(A,R)/(y_1,\dots,y_{i-1}).$$  So we have for any $1<i\leq d-3$,  $$\omega_{A/(x,\underline{y_{i-1}})}/y_i\omega_{A/(x,\underline{y_{i-1}})}\cong  \omega_{A/(x,\underline{y_i})}.$$ 
 
Since $\omega_{A/(x,\underline{y_i})}$ is liftable to $A/(x,\underline{y_{i-1}})$,  by   \cite[Proposition 2.4]{lifting}, $\Omega_{A/(x,\underline{y_i})}$ is also liftable to $A/(x,\underline{y_{i-1}})$  thus we have  $\Omega_{A/xA}/(\underline{y})\cong \Omega_{A/(x,\underline{y})}$. 

 %\begin{align}\label{exact_3}
%	0\to \omega_{A/(y_1,\dots,y_{i-1})}/y_i\to \omega_{A/(y_1,\dots,y_{i})}\to \operatorname{Ext} _R^{i}(A/(y_1,\dots,y_{i-1}),R) \xrightarrow{y_i} \operatorname{Ext} _R^{i}(A/(y_1,\dots,y_{i-1}),R) 
%\end{align}
%where the last map is injective. Thus $\omega_{A/(y_1,\dots,y_{i})}\cong A/(y_1,\dots, y_i)$ and $A/(y_1,\dots, y_i)
%$ is quasi-Gorenstein for any $1\leq i\leq d-3$. Thus we can reduce to step 2. 

Therefore, we may assume that $A$ is a complete $3$-dimensional quasi-Gorenstein local ring with $\operatorname{depth} A=2$ and  $H_m^2(A) =k$, we want to show that $\Omega_{A/xA}$ is  a maximal Cohen-Macaulay $A$-module. Let $R$ be a Gorenstein ring with $\operatorname{dim} R=\operatorname{dim} A$ and $A$ is a homomorphic image of $R$. We first prove the following claim.
\begin{claim}\label{claim}
	$\operatorname{Ext} ^2_A(\omega_{A/xA}, A)=0$. 
\end{claim}
\begin{proof}
	After applying $\operatorname{Hom}_R(-,R)$ to the short exact sequence $0\to A\xrightarrow{x} A\to A/xA\to 0$, we get the long exact sequence $$
	0\to A\xrightarrow{x} A \to \operatorname{Ext}^1_R(A/xA, R)\to \operatorname{Ext}^1_R(A, R)\xrightarrow{x}  \operatorname{Ext}^1_R(A, R)\to \dots$$ 
	By local duality \cite{hartshorne} and that $H_m^2(A)=k$, we have that $\operatorname{Ext}^1_R(A, R)\cong k $ and $k\xrightarrow{x} k$  is the zero map.

	  Thus we have the short exact sequence 
	\begin{equation}\label{short_exact_1}
		0\to A/xA\to \omega_{A/xA} \to k\to 0.
	\end{equation}
	(More generally, it follows from construction that $x\in \operatorname{ann}(\operatorname{Ext}_R^1(A,R))$ which gives the exact sequence
$
0\to A/xA\to \omega_{A/xA} \to H_m^2(A)^\vee \to 0.$)

	After applying $\operatorname{Hom}_{A/xA}(-, A/xA) $ to \eqref{short_exact_1}, we have the following exact sequence \eqref{exact1} and we want to show that $j$ in \eqref{exact1}  is surjective:
	\begin{align}\label{exact1}
		0&\to \operatorname{Hom}_{A/xA} (\omega_{A/xA},A/xA)\xrightarrow{i} \operatorname{Hom}_{A/xA} (A/xA,A/xA)\xrightarrow{j} \operatorname{Ext}_{A/xA}^1 (k, A/xA)\nonumber\\
		&\to \operatorname{Ext} _{A/xA}^1(\omega_{A/xA},A/xA)\to \operatorname{Ext}^1_{A/xA}(A/xA,A/xA)=0.
	\end{align}

	Moreover, after applying $\operatorname{Hom}_{A/xA}(k, -)$ to  \eqref{short_exact_1}, we have the exact sequence 
	$$0\to \operatorname{Hom}_{A/xA}(k,k)\cong k \to \operatorname{Ext} _{A/xA}^1(k,A/xA)\to 0$$
	where $\operatorname{Ext}_{A/xA}^i(k,\omega_{A/xA})=0$ for $i=0,1$ as $\operatorname{depth} \omega_{A/xA}\geq 2$. 
	Hence, $$\operatorname{Ext}_{A/xA}^1(k,A/xA)\cong k.$$
	%$\operatorname{Ext}_{A/xA}^1(k,A/xA)\cong \operatorname{Ext}_A ^2(k,A) \cong k$.
	\begin{comment}
	Since $\operatorname{depth} A=2$, there exists $y$ such that $x,y$ are regular on $A$,  $\operatorname{Ext}^1(k,A/xA)\cong \operatorname{Hom}(k, A/(x,y))\subset A/(x,y)$ and thus $\operatorname{rank} \operatorname{Ext}^1(k,A/xA)\leq 1$. Moreover, $\operatorname{depth} A/xA=1\Rightarrow \operatorname{Ext}^1(k,A/xA)\neq 0\Rightarrow \operatorname{Ext}^1(k,A/xA)=k$.
	\end{comment}
	
	Therefore, to show that $j$ is surjective, it suffices to show $j\neq 0$,  that is, $i$ is not surjective. If otherwise, $i$ is surjective then the existence of the preimage of $\mathrm{id}_{A/xA} $ implies that \eqref{short_exact_1} splits. Then $\omega_{A/xA} \cong A/xA\oplus k$ which is a contradiction as $\operatorname{depth} \omega_{A/xA}\geq 2$. Therefore, $j$ is surjective hence \begin{equation}
		\operatorname{Ext} ^2_A(\omega_{A/xA}, A)\cong \operatorname{Ext}^1_{A/xA} (\omega_{A/xA}, A/xA)=0.
	\end{equation}
	
\end{proof}

From \eqref{short_exact_1}, we have that $ \omega_{A/xA} $ is generated by at most $2$ elements. Since $A$ is $S_2$, $\omega_{A/xA} $ is  faithful over $A/xA$ \cite{canonical}.  Suppose $ \omega_{A/xA} $ is minimally generated by one element. Then $ \omega_{A/xA} \cong A/xA$  is Cohen-Macaulay which implies that $A$ is \cm, a contradiction. Therefore, $\omega_{A/xA} $ is minimally generated by $2$ elements, there is a short exact sequence 
\begin{equation}\label{exact_1}
	0\to \Omega_{A/xA}\to A^2\to \omega_{A/xA}\to 0. 
\end{equation}

After applying $\operatorname{Hom}_R(-,R)$ to \eqref{exact_1}, we have the exact sequence $$
\operatorname{Ext}_R^1(\omega_{A/xA},R) \xrightarrow{f} \operatorname{Ext} ^1_R(A^2,R)\to \operatorname{Ext} _R^1(\Omega_{A/xA},R) \to \operatorname{Ext} _R^2(\omega_{A/xA}, R)=0.$$ Here we have $\operatorname{Ext} _R^2(\omega_{A/xA}, R)=0$ since $\operatorname{depth} \omega_{A/xA} \geq 2$. Now to show that $\Omega_{A/xA}$ is maximal \cm, it suffices to show $\operatorname{Ext}_R^1(\Omega_{A/xA}, R)=H_m^2(\Omega_{A/xA})^\vee=0$ by local duality\cite{hartshorne}. Thus it suffices to show $\operatorname{Ext}_R^1(\omega_{A/xA},R) \xrightarrow{f} \operatorname{Ext} ^1_R(A^2,R)$ is surjective. 
	
For any $A$-module $M$, from the Grothendieck spectral sequence  \cite[P.145]{weibel}
$$\operatorname{Ext}_A^i(M, \operatorname{Ext}_R^j(A,R))\Rightarrow \operatorname{Ext}_R^{i+j} (M,R),$$ we have the five term exact sequence:
$$0\to \operatorname{Ext} _A^1(M,A) \to \operatorname{Ext} _R^1(M,R)\to \operatorname{Hom}_A(M, \operatorname{Ext}_R^1(A,R) ) \to  \operatorname{Ext} ^2_A(M,\operatorname{Hom}_R(A,R))\to \operatorname{Ext}_R^2(M,R) $$

which is the same as 
\begin{equation}\label{five_term_exact sequence}
	0\to \operatorname{Ext} _A^1(M,A) \to \operatorname{Ext} _R^1(M,R)\to \operatorname{Hom}_A(M, k ) \to  \operatorname{Ext} ^2_A(M,A)\to \operatorname{Ext}_R^2(M,R) 
\end{equation}

Take $M=A^2$ and $M=\omega_{A/xA}$ respectively, we have the following commutative diagram % https://q.uiver.app/?q=WzAsNCxbMCwwLCJFeHQgX1JeMShBXjIsUikiXSxbMSwwLCJIb21fQShBXjIsaykiXSxbMCwxLCJFeHRfUl4xKFxcb21lZ2Ffe0EveEF9LCBSKSJdLFsxLDEsIkhvbV9BKFxcb21lZ2Ffe0EveEF9LGspIl0sWzAsMSwiXFxnYW1tYSJdLFsyLDBdLFszLDEsIlxcYmV0YSIsMix7InN0eWxlIjp7ImhlYWQiOnsibmFtZSI6ImVwaSJ9fX1dLFsyLDMsIlxcYWxwaGEiLDIseyJzdHlsZSI6eyJoZWFkIjp7Im5hbWUiOiJlcGkifX19XV0=

\begin{center}
\begin{tikzcd}
	{\operatorname{Ext} _R^1(A^2,R)} & {\operatorname{Hom}_A(A^2,k)} \\
	{\operatorname{Ext}_R^1(\omega_{A/xA}, R)} & {\operatorname{Hom}_A(\omega_{A/xA},k)}
	\arrow["\cong", from=1-1, to=1-2]
	\arrow["f",from=2-1, to=1-1]
	\arrow["\beta"',"\cong"  from=2-2, to=1-2]
	\arrow["\alpha"', two heads, from=2-1, to=2-2]
\end{tikzcd}
\end{center}
where $\alpha$ is surjective since $\operatorname{Ext}^2_A(\omega_{A/xA},A)=0$ by Claim \ref{claim}. And $\beta$ is an isomorphism  since the map $A^2\to \omega_{A/x}$ induces an isomorphism $A^2\otimes A/m\to \omega_{A/x} \otimes A/m$ and the map $\beta$ can be identified as $\operatorname{Hom}_A(\omega_{A/x}\otimes A/m,k)\to \operatorname{Hom}_A(A^2\otimes A/m,k)$ which is an isomorphism. Therefore, $f$ is surjective and we finished the proof. \end{proof}

	\begin{remark}
		We see that for a quasi-Gorenstein local ring $(A,m,k)$ that is a homomorphic image of a Gorenstein ring $R$ with $\operatorname{depth}A=\operatorname{dim}A-1$, if $\operatorname{dim }A=3$, then $H_m^{\operatorname{dim}A-1}(A) $ must be of finite length. If $\operatorname{dim }A>3$, $H_m^{\operatorname{dim}A-1}(A) $ is not of finite length. Indeed, $\operatorname{depth} H_{m}^{n-1}(A)^\vee =\operatorname{depth}\operatorname{Ext}_R^1(A,R)=\operatorname{dim}A-3$ as shown in the introduction, or by \cite[Lemma 1]{canonical}.  	
	%	Note that for a ring $(A,m,k)$ satisfying the assumption of  the theorem, if $\operatorname{dim} A>3$, then $H_m^{\operatorname{dim}A-1}(A)$ is \underline{not} of finite length, and $\operatorname{depth} H_{m}^{n-1}(A)^\vee =n-3$ by \cite[Lemma 1]{canonical}.
	\end{remark}

\end{document}